\documentclass[12pt]{amsart}

\usepackage[T1]{fontenc}
\usepackage{lmodern}
\usepackage{amsmath,amssymb,amsthm,mathtools,booktabs,array}
\mathtoolsset{showonlyrefs=true}

\usepackage[a4paper,left=2cm,right=2cm,top=2cm,bottom=2cm]{geometry}
\usepackage{cite}

\usepackage{mathrsfs}

\usepackage[expansion=false]{microtype}

\usepackage{hyperref}
\hypersetup{colorlinks=true,linkcolor=blue,citecolor=blue,urlcolor=blue,
  pdftitle={Adams filtrations and secondary coherence for fibers of Steenrod operations},
  pdfauthor={Dang Vo Phuc}}

\theoremstyle{plain}
\newtheorem{theorem}{Theorem}[section]
\newtheorem{proposition}[theorem]{Proposition}
\newtheorem{lemma}[theorem]{Lemma}
\newtheorem{corollary}[theorem]{Corollary}
\theoremstyle{definition}

\newtheorem{remark}[theorem]{Remark}
\numberwithin{equation}{section}
\newcommand{\F}{\mathbb F_2}
\newcommand{\Z}{\mathbb Z}
\newcommand{\A}{\mathcal A}
\newcommand{\B}{\mathcal B}
\newcommand{\Sq}{\operatorname{Sq}}
\newcommand{\Ext}{\operatorname{Ext}}
\newcommand{\Hom}{\operatorname{Hom}}
\newcommand{\im}{\operatorname{im}}
\newcommand{\coker}{\operatorname{coker}}
\newcommand{\fib}{\operatorname{fib}}
\newcommand{\id}{\operatorname{id}}
\newcommand{\HZ}{H\mathbb Z}
\newcommand{\HB}{H^*_{\B}}
\newcommand{\sExt}{\operatorname{Ext}_{\B}}
\newcommand{\gr}{\operatorname{gr}}

\newcommand{\Fil}{\mathcal F}

\title[Filtrations and coherence for operation fibers]{Adams filtrations and secondary coherence\\ for fibers of Steenrod operations}
\author{\DJ\d{\u a}ng V\~o Ph\'uc}
\address{Department of Mathematics, FPT University, 
  An Phu Thinh New Urban Area, Quy Nhon, Vietnam}
\email{dangphuc150488@gmail.com}
\thanks{ORCID: \url{https://orcid.org/0000-0002-6885-3996}}
\date{}

\begin{document}

\begin{abstract}
The Adams spectral sequence for fibers of Steenrod operations depends on structural data beyond primary cohomology. Motivated by the recent Bruner--Rognes two-extension theorem and open questions regarding the direct secondary cohomology modeling of such fibers, we develop a general framework to compute their Adams filtrations and delineate the underlying algebraic obstructions. Let $f:Y\to\prod_\lambda\Sigma^{d_\lambda}H\F$ be a locally finite family of positive-degree stable operations, where $Y=H\F$ or $H\Z$; in the integral case assume $d_\lambda\geq2$. In the present paper, we determine the Adams filtration of its fiber from the minimal generators of $\im(f^*)$ as a left Steenrod module. An explicit target automorphism separates a minimal subfamily from the redundant coordinates and gives a noncanonical stable splitting. Combining this reduction with the Bruner--Rognes two-extension theorem and an acyclic cancellation gives $E_3=E_\infty$. In positive homotopy degrees, the filtration-one subgroup is the image of the dual of the induced map on minimal-generator quotients, and the redundant coordinates account for filtration zero. We also specify the zero-stem filtration and its possible cyclic extension of order four. A terminating Adem recursion realizes the splitting for the full positive even-square family and its finite truncations. Separately, to address the secondary modeling problem, we identify two fundamental algebraic obstructions to directly constructing the secondary cohomology of these fibers. First, we prove that whenever the primary operation row has a nonzero kernel, its cokernel is not isomorphic to the cohomology of the fiber, a discrepancy that cannot be resolved by higher chain modifications. Second, for fixed lifts of a primary complex in a crossed bimodule, we construct a class in the third cohomology of an explicit Hom complex whose vanishing is equivalent to the existence of coherent null-composite tracks. These obstructions demonstrate why the secondary Ext groups of these fibers must be derived from the topological $d_2$ via the Baues--Jibladze comparison theorem, rather than through a direct secondary algebraic resolution. 
\end{abstract}

\maketitle

\section{Introduction}\label{sec:introduction}

The mod-$2$ Adams spectral sequence measures stable homotopy through
successive approximations by ordinary cohomology. Its second page is
governed by the Steenrod algebra, but the passage from that page to
the associated graded homotopy groups depends on information that
primary cohomology alone does not display. For the spectra considered
here, the abutment is understood after $2$-completion. Adams's
foundational work \cite{Adams1958} provides the historical setting
for this distinction. The issue becomes particularly concrete for
fibers of maps between Eilenberg--Mac Lane spectra: the homotopy
long exact sequence constrains the underlying groups, whereas their
Adams filtrations need not be visible from a presentation of the
primary cohomology module. 

Bruner and Rognes identified a specified component of the Adams
$d_2$-differential with Yoneda composition by the two-extension
associated to a fiber sequence
\cite[Theorem~1.1]{BrunerRognes2022}. Bruner \cite{Bruner2026} applied this theorem
to determine the $E_3=E_\infty$ terms for individual-square fibers
over $H\F$ and $H\Z$, and for the full positive even-square family
over $H\Z$. The theory of projective and
injective classes in triangulated categories provides a broader
framework for Adams spectral sequences and their filtration bounds
\cite{Christensen1998}.

Secondary cohomology addresses another part of the problem. Baues
constructed the algebra of secondary cohomology operations
\cite{Baues2006}, and Baues--Jibladze identified the secondary
differential with the Adams $d_2$ under the comparison isomorphisms,
thereby identifying secondary Ext with the Adams $E_3$-term
\cite[Theorem~7.3]{BauesJibladze2006}. Nassau's smaller crossed
algebra provides explicit coordinates at the prime $2$
\cite{Nassau2012}. In \cite{Bruner2026}, Bruner observes that his $E_3$ calculations
already determine the corresponding secondary Ext groups. His
Question~6.1 asks instead for a computation of the secondary
cohomology of the fibers from which those $E_3$-terms can be
derived. In addition, his Question~6.2 asks for a proof or improvement of the
two-extension theorem using secondary cohomology and the secondary
Adams spectral sequence \cite[Section~6]{Bruner2026}. These
questions distinguish the value of a secondary derived functor
from an explicit construction of the secondary object and its
resolution.

For the Adams-filtration calculation, redundancy is relative to the
chosen family of operations. An operation is redundant when its
image lies in the left Steenrod module generated by the images of
the other operations. An operation that is decomposable in the
Steenrod algebra can therefore remain a minimal generator of the
image module of a particular family. A secondary construction
requires additional data. Null tracks for adjacent composites must
satisfy coherence, and an augmented secondary complex must satisfy
the augmentation track equation. Its underlying primary complex
must also resolve the correct cohomology module. A primary relation
supplies a prescribed null-composite boundary but does not by itself
establish these further conditions. In this paper, we determine the operation-fiber Adams filtration
from relative minimal generators. Crucially, to address the direct secondary modeling problem posed by Bruner's open Questions 6.1 and 6.2 \cite{Bruner2026}, we delineate the exact algebraic obstructions to these secondary liftings. We briefly describe our main results as follows.

\begin{itemize}
\item For a locally finite family of operations on $H\F$ or $H\Z$
satisfying the degree hypotheses of Section~\ref{sec:rows}, we
construct a stable splitting that separates a minimal subfamily
from the redundant target coordinates and determine the
corresponding $E_3=E_\infty$ term. In positive homotopy degrees,
we identify the subgroup of Adams filtration at least one with
the image of the dual map induced by the operation row on
minimal-generator quotients. We also specify the zero-stem
filtration, including its possible extension of order four; see
Theorems~\ref{thm:splitting} and~\ref{thm:general-adams} and
Corollary~\ref{cor:zero-stem}.

\item We derive a terminating right-factor recursion expressing
every $\Sq^m$ with $m>0$ not a power of $2$ as a sum of left
Steenrod multiples of lower power-of-two squares. After passage to
integral cohomology, this recursion produces an explicit target
involution for the full positive even-square family and its finite
truncations; see Theorem~\ref{thm:even-splitting}.

\item We compute the kernel and cokernel of the relevant Yoneda
composite and justify the acyclic cancellation used in the Adams
calculation. The Baues--Jibladze comparison theorem then determines
the associated secondary Ext groups as bigraded vector spaces.
For Bruner's square families, the resulting formulas recover his
calculations with their original attribution; see
Section~\ref{sec:examples}.

\item We identify two fundamental algebraic obstructions to the direct secondary modeling of these fibers. First, we establish a necessary and sufficient coherence criterion for fixed differential lifts in a crossed algebra by constructing an explicit obstruction class in the third cohomology of a Hom complex. Second, we prove that whenever the primary operation row has a nonzero kernel, its cokernel is not isomorphic to the cohomology of the fiber, and that changing only higher chain degrees cannot correct this discrepancy. These structural results demonstrate why the secondary Ext groups of these fibers must be derived topologically rather than via a direct algebraic construction; see Theorem~\ref{thm:coherence-obstruction} and Proposition~\ref{prop:augmentation-obstruction}.
\end{itemize}

The Adams-filtration argument begins with the graded quotient
$I/\A^+I$, where $I$ is the image of the operation map and $\A^+$
is the positive-degree ideal of the Steenrod algebra. A homogeneous
choice of minimal generators determines a target shear that makes
the redundant coordinates zero. For the resulting minimal family,
we combine the injectivity of the ordinary Ext connecting
homomorphisms with the Bruner--Rognes formula \cite{BrunerRognes2022} to establish the
required Ext exactness. We then remove an explicit acyclic
subcomplex, determine the induced differential on the quotient by
bidegree and naturality, and identify the filtration subgroup
inside homotopy. The secondary analysis is conducted separately
in a crossed bimodule. Normalized set-theoretic sections supply
exact boundary representatives, and the Peiffer identity shows
that the coherence defects form a cocycle. Changing the null
tracks changes this cocycle by a coboundary. 

\subsection*{Organization of the paper}

Section~\ref{sec:primary} fixes the grading conventions, establishes
the primary resolution of integral cohomology, and derives the
terminating right-factor recursion for the Steenrod squares.
Section~\ref{sec:rows} constructs the minimal and redundant parts
of a general operation family and realizes their separation by an
automorphism of the target spectrum. Section~\ref{sec:adams}
proves the general Adams-filtration theorem, including the required
Ext exactness, the acyclic cancellation, the exclusion of residual
differentials, and the identification of the filtration subgroup
inside positive-degree homotopy. It also determines the zero-stem
filtration and the maps induced by passage to subfamilies.
Section~\ref{sec:examples} applies the theorem to individual
squares, the full positive even-square family, finite truncations,
and families containing redundant operations; it also records
the resulting secondary Ext comparison.
Section~\ref{sec:tracks} establishes the exact lift conventions
for a crossed algebra and proves the coherence obstruction
criterion for fixed differential lifts.
Finally, Section~\ref{sec:augmentation} proves the cokernel obstruction
for operation rows with nonzero kernel, states the augmentation
track equation, and clarifies why these obstructions necessitate the topological derivation of secondary Ext groups over a direct algebraic secondary construction.

\section{Primary algebra and a recursion for the squares}\label{sec:primary}

We first establish the primary algebra used to separate the target
coordinates. The integral cohomology module supplies the
zero-stem tower, while the Adem recursion identifies which
single-square columns can be eliminated. The arguments also fix the
distinction between a left module relation and a choice of a lift of
that relation to a crossed algebra.

Let $H=H\F$, let $\A$ be the mod-$2$ Steenrod algebra, and let
$\A^+=\bigoplus_{d>0}\A^d$. Write $\epsilon:\A\to\F$ for the
augmentation, which is the identity in degree zero and is zero
in positive degrees. All cohomology groups have coefficients in
$\F$ unless other coefficients are displayed. All modules in the
primary calculation are graded left $\A$-modules. We use
\[
 \Ext^{s,t}_{\A}(M,\F)
 =\Ext^s_{\mathrm{gr}\text{-}\A}(M,\Sigma^t\F),
 \qquad (\Sigma^r M)^d=M^{d-r}.
\]
In particular,
\begin{equation}\label{eq:shift}
 \Ext^{s,t}_{\A}(\Sigma^{-1}M,\F)
 =\Ext^{s,t+1}_{\A}(M,\F).
\end{equation}
The Adams differential has bidegree $(r,r-1)$, and the stem of
$(s,t)$ is $t-s$. For a vector space $W$, the notation
$\Sigma^{s,t}W$ places $W$ in the single bidegree $(s,t)$.
This differs from the single grading shift $\Sigma^r M$ of a
module or spectrum. A superscript $\vee$ on a finite-dimensional
vector space denotes its ordinary linear dual; any internal degree
assigned to that dual will be stated explicitly. We write
$\widehat\pi_kX=\varprojlim_r\pi_kX/2^r\pi_kX$ for the
abelian-group $2$-adic completion. The decreasing Adams filtration
on this abutment is denoted by $\Fil^s\widehat\pi_kX$, and
$\gr^s=\Fil^s/\Fil^{s+1}$. In positive degrees of the fibers
considered below, completion does not change the homotopy groups.

Write
\[
 B=H^*(\HZ;\F)=\A/\A\Sq^1.
\]
Here $\A\Sq^1$ is a left ideal. A left-$\A$-linear map with
$1\mapsto\Sq^n$ sends $a$ to $a\Sq^n$, so it is right
multiplication by $\Sq^n$.

\begin{lemma}\label{lem:integral-resolution}
Let $B=\A/\A\Sq^1$. Then the augmented sequence
\[
 \cdots\longrightarrow\Sigma^2\A
 \xrightarrow{\,\cdot\Sq^1\,}\Sigma\A
 \xrightarrow{\,\cdot\Sq^1\,}\A\longrightarrow B
 \longrightarrow0
\]
is a free resolution. Consequently $\Ext^{s,t}_{\A}(B,\F)$ is
one-dimensional if $s=t\geq0$ and is zero otherwise.
\end{lemma}

\begin{proof}
An admissible Steenrod monomial is a product
$\Sq^{i_1}\cdots\Sq^{i_k}$ with
$i_j\geq2i_{j+1}$ and $i_k>0$; the empty product is also included.
The admissible basis is partitioned into words ending in $\Sq^1$
and words not ending in $\Sq^1$. Right multiplication by $\Sq^1$
sends a nonempty word of the second kind to the admissible word
obtained by appending $\Sq^1$, because its last exponent is at
least $2$. It sends the empty word to $\Sq^1$. These images are
distinct and exhaust the words ending in $\Sq^1$. On a word of
the first kind, right multiplication gives zero because
$\Sq^1\Sq^1=0$. Therefore the kernel and the image of right
multiplication by $\Sq^1$ coincide, and both equal
$\A\Sq^1$. This proves exactness in every positive resolution
degree and at $\A$ after augmentation. Applying graded
$\Hom_{\A}(-,\Sigma^t\F)$ gives zero differentials, since
$\Sq^1$ acts trivially on $\F$. The generator of $\Sigma^s\A$
has internal degree $s$, so its graded dual occurs precisely when
$t=s$.
\end{proof}

We denote this bigraded vector space by $\F[h_0]$, with
$|h_0|=(1,1)$. This notation will record the tower and does not
assert a ring structure on the Ext groups of an arbitrary fiber.

\begin{lemma}\label{lem:connectivity-ext}
Let $M$ be a bounded-below, locally finite graded left
$\A$-module, and let $a$ be an integer. If $M^d=0$ for $d<a$,
then $\Ext^{s,t}_{\A}(M,\F)=0$ for $t<a+s$.
\end{lemma}

\begin{proof}
Choose homogeneous lifts of a basis of $M/\A^+M$ and let
$P_0\to M$ be the free map determined by those lifts. Its
surjectivity follows by induction on degree. Indeed, after
subtracting the chosen lifts in the degree of an element, the
remainder is a sum of positive-degree operations on elements of
smaller degree. The induction begins below the lower bound of $M$.
There are finitely many chosen generators in each degree.
Since $\A$ is locally finite and $M$ is bounded below, $P_0$
is locally finite as well. Its kernel is therefore bounded below
and locally finite, so the same construction can be repeated.
The kernel of this free map is contained in $\A^+P_0$: a relation
with a nonzero scalar part would give a nonzero relation among the
chosen basis elements of $M/\A^+M$. Thus $P_0$ has generators in
degrees at least $a$ and its kernel is concentrated in degrees at
least $a+1$. Repeating the same construction on the kernel produces
a free resolution $P_\bullet$ with all generators of $P_s$ in
degrees at least $a+s$. In the cochain complex obtained by applying
$\Hom_{\A}(-,\Sigma^t\F)$, the term in degree $s$ is therefore
zero if $t<a+s$. Its cohomology is zero in that range as well.
\end{proof}

For $0<a<2b$, the Adem relations \cite{Adem1952} are
\begin{equation}\label{eq:adem}
 \Sq^a\Sq^b
 =\sum_{k=0}^{\lfloor a/2\rfloor}
 \binom{b-k-1}{a-2k}\Sq^{a+b-k}\Sq^k.
\end{equation}
Coefficients are read in $\F$, and $\Sq^0=1$.

\begin{lemma}\label{lem:binary-binomial}
Let $N,r$ be nonnegative integers. Then $\binom Nr$ is odd if
and only if every binary digit equal to $1$ in $r$ is also equal
to $1$ in $N$.
\end{lemma}

\begin{proof}
Write $N=\sum_i\epsilon_i2^i$, where
$\epsilon_i\in\{0,1\}$. In $\F[x]$, repeated squaring gives
$(1+x)^{2^i}=1+x^{2^i}$. Consequently
\[
 (1+x)^N=\prod_{\epsilon_i=1}(1+x^{2^i}).
\]
The coefficient of $x^r$ in the product is $1$ precisely when
$r$ is the sum of a subset of the distinct powers $2^i$ that
occur in $N$. Uniqueness of binary expansion gives the stated
condition. The coefficient of $x^r$ on the left is
$\binom Nr$ reduced modulo $2$.
\end{proof}

\begin{proposition}\label{prop:square-recursion}
Let $n>0$. Then $\Sq^n$ has a nonzero class in
$\A^+/(\A^+)^2$ if and only if $n$ is a power of $2$.
If $n$ is not a power of $2$, write $n=q+r$, where $q$ is
the largest power of $2$ smaller than $n$ and $0<r<q$. Then
\begin{equation}\label{eq:square-recursion}
 \Sq^n=\Sq^r\Sq^q+
 \sum_{k=1}^{\lfloor r/2\rfloor}
 \binom{q-k-1}{r-2k}\Sq^{n-k}\Sq^k.
\end{equation}
Recursively expanding the rightmost square in the terms of
\eqref{eq:square-recursion} gives a finite expression
\begin{equation}\label{eq:right-power-expansion}
 \Sq^n=\sum_{2^j\leq n}a_{n,j}\Sq^{2^j},
 \qquad a_{n,j}\in\A^{n-2^j},
\end{equation}
in which only powers $2^j<n$ occur when $n$ is not a power of $2$.
\end{proposition}

\begin{proof}
Apply \eqref{eq:adem} with $a=r$ and $b=q$. The coefficient
of the $k=0$ term is $\binom{q-1}{r}=1$ in $\F$, because
every binary digit below the leading digit of $q$ occurs in
$q-1$. Moving that term to the other side gives
\eqref{eq:square-recursion}. Its right side is a sum of products
of positive-degree operations. This proves decomposability when
$n$ is not a power of $2$.

For the converse, use the presentation of $\A$ by the generators
$\Sq^m$, $m>0$, and the Adem relations. Passing to
$\A^+/(\A^+)^2$ deletes every product of two positive-degree
generators. The only potentially nonzero linear term in an Adem
relation is $\binom{b-1}{a}\Sq^{a+b}$. If $a+b=2^j$,
then $b-1=(2^j-1)-a$. Among the $j$ binary digits below $2^j$,
the digits of $b-1$ are the complements of those of $a$.
Since $a>0$, at least one digit of $a$ is $1$ and the
corresponding digit of $b-1$ is $0$. By
Lemma~\ref{lem:binary-binomial}, $\binom{b-1}{a}$ is even.
Thus no defining relation has a nonzero linear component in
degree $2^j$. Multiplying any defining relation by a
positive-degree word cannot create a linear term. Hence the
class of $\Sq^{2^j}$ survives in the indecomposable quotient.

It remains to verify the recursion. For a power of $2$, take
$a_{2^j,j}=1$ and all other coefficients zero. For a nonpower
$n=q+r$, the first term of \eqref{eq:square-recursion}
already ends in the power-of-two square $\Sq^q$. Every
remaining rightmost exponent satisfies
$1\leq k\leq\lfloor r/2\rfloor<n$. Expanding those factors
using the previously constructed expressions therefore gives
\eqref{eq:right-power-expansion} by induction on $n$.
There are finitely many terms at each step, and the exponent
to which the recursion is applied strictly decreases, so the
recursion terminates. Homogeneity gives
$|a_{n,j}|=n-2^j$; all powers occurring for a nonpower $n$
are strictly smaller than $n$.
\end{proof}

In particular, the images of $\Sq^{2^j}$, $j>0$, generate
$B^{>0}$ as a left $\A$-module. To verify generation,
represent a positive-degree class of $B$ by a sum of
positive-length admissible words in $\A$. A word ending
in $\Sq^1$ is zero in $B$. In every remaining word,
apply \eqref{eq:right-power-expansion} to the rightmost
square. The rightmost power $\Sq^1$, if it occurs,
again gives zero in $B$; all other terms are left
multiples of the stated generators. In the expression for an
even nonpower $m$, the term $a_{m,0}\Sq^1$ vanishes in
$B$. We obtain
\begin{equation}\label{eq:triangular-B}
 \overline{\Sq^m}=
 \sum_{0<2^j<m}a_{m,j}\overline{\Sq^{2^j}}.
\end{equation}
The classes of these power-of-two squares are independent in
$B^{>0}/\A^+B^{>0}$. A relation in this quotient would give,
in $\A$, a relation among their indecomposable classes and
decomposable terms. Proposition~\ref{prop:square-recursion}
excludes such a relation.

\section{Minimal operation families and stable splittings}\label{sec:rows}

The primary recursion identifies redundant squares in the full even
family. We now treat redundancy for an arbitrary family of stable
operations. The quotient by positive-degree multiples determines a
minimal family, and its choice gives an actual equivalence of target
spectra. This reduction will permit the Adams calculation to be
proved first for a minimal family.

Let $Y=H$ or $Y=\HZ$, and in this section write $B_Y=H^*Y$.
Thus $B_H=\A$ and $B_{\HZ}=\A/\A\Sq^1$. Let
$(\theta_\lambda)_{\lambda\in\Lambda}$ be a family with
$\theta_\lambda\in B_Y^{d_\lambda}$. Assume $d_\lambda\geq1$
if $Y=H$ and $d_\lambda\geq2$ if $Y=\HZ$. Assume also that
\begin{equation}\label{eq:local-finiteness}
 \#\{\lambda\in\Lambda:d_\lambda\leq d\}<\infty
 \quad\text{for every integer }d.
\end{equation}
The zero operation is allowed. The corresponding map and fiber are
\[
 f:Y\longrightarrow G=\prod_{\lambda\in\Lambda}
       \Sigma^{d_\lambda}H,
 \qquad X_f=\fib(f).
\]
The cohomology map is the row
\begin{equation}\label{eq:general-row}
 \nu:P=\bigoplus_{\lambda\in\Lambda}\A e_\lambda
       \longrightarrow B_Y,
 \qquad |e_\lambda|=d_\lambda,
 \qquad \nu(ae_\lambda)=a\theta_\lambda.
\end{equation}
Set
\begin{equation}\label{eq:IKC-VJR}
 \begin{gathered}
 I=\im\nu,\qquad K=\ker\nu,\qquad C=B_Y/I,\\
 V=P/\A^+P,\qquad J=I/\A^+I,
 \qquad R=\ker(V\longrightarrow J).
 \end{gathered}
\end{equation}
The map $V\to J$ is induced by $\nu$ and is surjective.
We call the family minimal if this map is an isomorphism.
The classes of the $\theta_\lambda$ span $J$; in each degree
one can select a basis from these classes. The union of the
selected indices gives a subset $S\subseteq\Lambda$ whose
image classes form a homogeneous basis of $J$. This choice is
not asserted to be unique.

The product defining $G$ is equivalent to the corresponding wedge.
Indeed, in stable homotopy degree $q$, only factors with
$d_\lambda=q$ contribute, and there are finitely many such factors.
The canonical map from the wedge to the product is consequently an
isomorphism on every stable homotopy group. Applying cohomology
to the wedge gives a product of suspended copies of $\A$; in
each cohomological degree that product is finite by
\eqref{eq:local-finiteness}. This proves \eqref{eq:general-row}
without introducing a completed free module.

\begin{lemma}\label{lem:minimal-generators}
Let $S\subseteq\Lambda$ be such that the classes of
$\theta_s$, $s\in S$, form a homogeneous basis of $J$.
Then the $\theta_s$ generate $I$ as a left $\A$-module,
and the kernel of
$\bigoplus_{s\in S}\A e_s\to I$
is contained in $\A^+\bigoplus_{s\in S}\A e_s$.
\end{lemma}

\begin{proof}
The degree induction in the proof of
Lemma~\ref{lem:connectivity-ext} applies to any homogeneous lifts
of a basis of the quotient by positive-degree multiples. Applied
to $I$, with the specified lifts $\theta_s$, it proves the
generation assertion. If $\sum_s a_s e_s$ is in the kernel,
reducing its image modulo $\A^+I$ gives
$\sum_s\epsilon(a_s)[\theta_s]=0$ in $J$.
Independence of the specified basis gives $\epsilon(a_s)=0$
for every $s$, so every coefficient belongs to $\A^+$.
\end{proof}

\begin{theorem}\label{thm:splitting}
Let $f$ satisfy the preceding hypotheses, and let $S$ be as in
Lemma~\ref{lem:minimal-generators}. Put
\[
 G_S=\prod_{s\in S}\Sigma^{d_s}H,
 \qquad G_T=\prod_{r\in\Lambda\setminus S}\Sigma^{d_r}H,
 \qquad X_S=\fib(f_S:Y\to G_S).
\]
Then there is a stable equivalence
\begin{equation}\label{eq:general-splitting}
 X_f\simeq X_S\times\Omega G_T.
\end{equation}
For each $r\notin S$, choose a finite homogeneous expression
\begin{equation}\label{eq:coefficients}
 \theta_r=\sum_{s\in S}a_{rs}\theta_s,
 \qquad a_{rs}\in\A^{d_r-d_s}.
\end{equation}
These expressions define a map $L:G_S\to G_T$ and an
involution
\begin{equation}\label{eq:shear}
 T:G_S\times G_T\longrightarrow G_S\times G_T,
 \qquad T(z,w)=(z,w+Lz),
\end{equation}
such that $Tf\simeq(f_S,0)$; this involution induces
\eqref{eq:general-splitting}. The number of factors of
$\Omega G_T$ in suspension degree $d-1$ is $\dim R_d$.
\end{theorem}

\begin{proof}
Lemma~\ref{lem:minimal-generators} gives an expression
\eqref{eq:coefficients} for every redundant operation. Since the
row is homogeneous, taking degree-$d_r$ components makes each
coefficient homogeneous of the indicated degree. No term with
$d_s>d_r$ occurs. A coefficient $a_{rs}$ represents a stable
operation $\Sigma^{d_s}H\to\Sigma^{d_r}H$. Compose it with the
$s$th projection from $G_S$ and take the finite sum over $s$.
This defines the $r$th component of $L$, and the universal
property of the product defines $L$ itself.

Each factor of $G_T$ has $2\id=0$, and equality of maps into
a product is detected by its projections. Thus $2\id_{G_T}=0$
in the stable homotopy category, and hence $L+L=0$. Substituting
\eqref{eq:shear} into itself gives
$T^2(z,w)=(z,w+Lz+Lz)=(z,w)$, so $T$ is an equivalence.
Equation~\eqref{eq:coefficients} says that every component of
$f_T+Lf_S$ is zero in $[Y,\Sigma^{d_r}H]$.
Consequently $Tf\simeq(f_S,0)$.

An equivalence of targets induces an equivalence of homotopy
fibers. The homotopy fiber of $(f_S,0)$ is the product
$\fib(f_S)\times\Omega G_T$: the null-homotopy of its second
component is an arbitrary based loop in $G_T$. Choosing the
null-homotopies just established gives
\eqref{eq:general-splitting}.

In degree $d$, the coordinate generators form a basis of $V_d$,
and the coordinates selected by $S$ map to a basis of $J_d$.
The exact sequence $0\to R_d\to V_d\to J_d\to0$ therefore
gives
\[
 \#\{r\notin S:d_r=d\}=\dim V_d-\dim J_d=\dim R_d.
\]
This is the stated multiplicity. The choice of $S$, the
coefficients, and the homotopies can change the displayed
equivalence; no canonical splitting is asserted.
\end{proof}

\section{The two-extension and the Adams filtration}\label{sec:adams}

The target splitting reduces the calculation to a family whose
images are minimal module generators. We now prove that its
two-extension composite is injective, use Bruner--Rognes's theorem \cite{BrunerRognes2022} to obtain
the required exact sequence in Ext, and remove an acyclic
subcomplex. The remaining differential is checked explicitly.
This order of argument avoids assuming a direct-sum description
of the Adams $E_2$-term. After identifying the positive-degree
filtration, we determine the zero-stem extension and then describe
the maps induced by target projections.

For a fiber sequence $X\to Y\xrightarrow{f}G$, let
$P$, $I$, $K$, and $C$ be as in \eqref{eq:IKC-VJR}.
The long exact cohomology sequence gives
\begin{equation}\label{eq:three-ses}
 \begin{gathered}
 0\longrightarrow K\longrightarrow P\longrightarrow I
       \longrightarrow0,\\
 0\longrightarrow I\longrightarrow B_Y\longrightarrow C
       \longrightarrow0,\\
 0\longrightarrow C\xrightarrow{i}H^*X
       \xrightarrow{q}\Sigma^{-1}K\longrightarrow0.
 \end{gathered}
\end{equation}
We suppress the second argument $\F$ in ordinary Ext groups.
The first two sequences define
\begin{equation}\label{eq:D}
 \begin{aligned}
 D^{s,t}:\Ext^{s,t}(\Sigma^{-1}K)
 &=\Ext^{s,t+1}(K)
 \xrightarrow{\partial_{IK}}\Ext^{s+1,t+1}(I)\\
 &\xrightarrow{\partial_{CI}}\Ext^{s+2,t+1}(C).
 \end{aligned}
\end{equation}
Equivalently, $D$ is Yoneda composition with
\begin{equation}\label{eq:two-extension}
 0\longrightarrow K\longrightarrow P\xrightarrow{\nu}B_Y
       \longrightarrow C\longrightarrow0.
\end{equation}
To see this identification, the connecting homomorphism for
$0\to K\to P\to I\to0$ appends that extension to a
Yoneda representative with leftmost coefficient $\F$ and
rightmost argument $K$; the second connecting homomorphism
appends $0\to I\to B_Y\to C\to0$. Their splice is exactly
\eqref{eq:two-extension}. The internal shift is the one in
\eqref{eq:shift}.

The external differential theorem used below is
\cite[Theorem~1.1]{BrunerRognes2022}. At the prime $2$, it states
\begin{equation}\label{eq:BR}
 i^*d_2q^*=D.
\end{equation}
Equation~\eqref{eq:BR} specifies a composite through the
Adams $d_2$ of $X$. It does not identify an arbitrary entire
differential with a single block of a chosen decomposition.

\begin{lemma}\label{lem:D-general}
Let $\nu$ be an operation row satisfying the hypotheses of
Section~\ref{sec:rows}. Then $\partial_{CI}$ is injective,
and $\partial_{IK}$ is an isomorphism on
$\Ext^{s,t}(K)$ for $s>0$. On filtration zero there is an
exact sequence
\begin{equation}\label{eq:restriction-sequence}
 0\longrightarrow J_t^\vee\longrightarrow V_t^\vee
 \longrightarrow\Ext^{0,t}(K)
 \xrightarrow{\partial_{IK}}\Ext^{1,t}(I)
 \longrightarrow0.
\end{equation}
Consequently
\begin{equation}\label{eq:D-general-kc}
 \begin{aligned}
 \ker D&\cong\bigoplus_{d\geq1}\Sigma^{0,d-1}R_d^\vee,\\
 \coker D&\cong\Ext^{*,*}(B_Y)
       \oplus\bigoplus_{d\geq1}\Sigma^{1,d}J_d^\vee.
 \end{aligned}
\end{equation}
The displayed isomorphisms are isomorphisms of bigraded
vector spaces. Here $J_d=R_d=0$ in degree $1$ for an integral
source, and the cokernel is graded by the target of
\eqref{eq:D}. If the family is minimal, then $D$ is injective.
\end{lemma}

\begin{proof}
The middle module $P$ is free, so $\Ext^{s,t}(P)=0$ for
$s>0$. Moreover,
$\Ext^{0,t}(P)=V_t^\vee$ and
$\Ext^{0,t}(I)=J_t^\vee$, since maps to the trivial module
annihilate positive-degree multiples. The long exact sequence
of the first sequence in \eqref{eq:three-ses} therefore gives
\eqref{eq:restriction-sequence} and the stated isomorphisms
in positive filtration. The map $J_t^\vee\to V_t^\vee$
is dual to the surjection $V_t\to J_t$. Its cokernel is
$R_t^\vee$. Thus the kernel of $\partial_{IK}$ is
$R_t^\vee$ in filtration zero and is zero in every
positive filtration.

If $Y=H$, the group $\Ext(B_Y)$ is supported only in
$(0,0)$ and $I$ has no degree-zero part. If $Y=\HZ$,
Lemma~\ref{lem:integral-resolution} places $\Ext(B_Y)$ on
the diagonal $t=s$, whereas
Lemma~\ref{lem:connectivity-ext} gives
$\Ext^{s,t}(I)=0$ for $t<s+2$. In both cases the map
$\Ext^{s,t}(B_Y)\to\Ext^{s,t}(I)$ is zero in every
bidegree. The second sequence in \eqref{eq:three-ses}
therefore gives, for $s\geq0$, the exact sequence
\begin{equation}\label{eq:CI-exact}
 0\longrightarrow\Ext^{s-1,t}(I)
 \xrightarrow{\partial_{CI}}\Ext^{s,t}(C)
 \longrightarrow\Ext^{s,t}(B_Y)\longrightarrow0,
\end{equation}
where the negative Ext group is zero. This proves
injectivity of $\partial_{CI}$.

It follows that the kernel of $D$ is the kernel of
$\partial_{IK}$ with the suspension shift in
\eqref{eq:shift}; this is the first formula in
\eqref{eq:D-general-kc}. The first connecting map is
surjective onto $\Ext^{s,t}(I)$ for every $s\geq1$.
Hence the image of $D$ is the image under
$\partial_{CI}$ of exactly these positive-filtration
groups. The part of \eqref{eq:CI-exact} that it leaves
unreached is $\Ext(B_Y)$ together with the injective
image of $\Ext^{0,d}(I)=J_d^\vee$, which occurs in
bidegree $(1,d)$. These latter bidegrees are disjoint
from the support of $\Ext(B_Y)$: for $Y=H$ the latter
has filtration zero, and for $Y=\HZ$ the allowed
degrees satisfy $d\geq2$. This gives the second formula.
If the family is minimal, $R=0$, so the kernel is zero.
\end{proof}

\begin{lemma}\label{lem:acyclic-cancellation}
Let $0\to U\xrightarrow{a}E\xrightarrow{b}W\to0$
be an exact sequence of bigraded vector spaces. Let
$d:E\to E$ have bidegree $(2,1)$ and satisfy $d^2=0$.
If $D=bda$ is injective, then
$S=aU+d(aU)$ is an acyclic subcomplex, and $b$ induces
an isomorphism of bigraded vector spaces
\[
 E/S\cong W/\im D.
\]
In particular, $H(E,d)\cong H(W/\im D,\bar d)$,
where $\bar d$ denotes the induced quotient differential.
\end{lemma}

\begin{proof}
Suppose $au=dav$. Applying $b$ gives $Dv=0$, so $v=0$
and then $u=0$. Thus $aU\cap d(aU)=0$. If $da(u)=0$,
then $D(u)=0$, so $u=0$; hence $d$ maps $aU$
isomorphically onto $d(aU)$. It vanishes on $d(aU)$
because $d^2=0$. This proves that $S$ is a subcomplex
with zero homology. Explicitly, the map that sends
$da(u)$ to $a(u)$ and sends $a(u)$ to zero is a
contracting homotopy on $S$.

The map $b$ sends $S$ onto $\im D$ and is surjective
on $E$. If $b(e)=Dv$, then
$b(e-da(v))=0$, so $e-da(v)\in aU$ and $e\in S$.
Therefore the induced map $E/S\to W/\im D$ is both
injective and surjective. The long exact homology
sequence of $0\to S\to E\to E/S\to0$, together
with $H(S)=0$, proves the last assertion. This
argument leaves the quotient differential to be
determined; it does not assume that it is zero.
\end{proof}

\begin{theorem}\label{thm:general-adams}
Let $f:Y\to G$ satisfy the hypotheses of
Section~\ref{sec:rows}, let $X_f=\fib(f)$, and let
$J$ and $R$ be defined by \eqref{eq:IKC-VJR}. Then
the mod-$2$ Adams spectral sequence of $X_f$ converges
to the degreewise $2$-adic completion of its homotopy
groups and collapses at $E_3$. There is an isomorphism
of bigraded vector spaces
\begin{equation}\label{eq:general-e3}
 \begin{aligned}
 E_3^{*,*}(X_f)=E_\infty^{*,*}(X_f)
 \cong{}&T_Y\oplus
       \bigoplus_{d\geq1}\Sigma^{1,d}J_d^\vee\\
 &\oplus\bigoplus_{d\geq1}\Sigma^{0,d-1}R_d^\vee,
 \qquad
 T_Y=\begin{cases}
 \F,&Y=H,\\
 \F[h_0],&Y=\HZ.
 \end{cases}
 \end{aligned}
\end{equation}
For every $d\geq2$, the homotopy connecting map
identifies $\pi_{d-1}X_f$ with $V_d^\vee$, and its
Adams filtration satisfies
\begin{equation}\label{eq:canonical-filtration}
 \begin{gathered}
 \Fil^0\pi_{d-1}X_f=V_d^\vee,\qquad
 \Fil^1\pi_{d-1}X_f
   =\im(J_d^\vee\longrightarrow V_d^\vee)=R_d^\perp,\\
 \Fil^2\pi_{d-1}X_f=0.
 \end{gathered}
\end{equation}
\end{theorem}

\begin{proof}
The spectrum $G$ is connective and has finite-dimensional
homotopy in each degree by \eqref{eq:local-finiteness}.
The fiber $X_f$ is connective because $G$ has no
degree-zero homotopy and $Y$ is connective. Its homotopy
long exact sequence shows that every $\pi_kX_f$ is
finitely generated. In positive degrees these groups
are finite elementary abelian $2$-groups. For $Y=H$
and target factors in degree $1$, the degree-zero
group is a finite extension of such a group by $\F$;
for the integral source the degree-zero group is $\Z$.
Apply \cite[Theorem~6.6]{Bousfield1979} with the ring
spectrum $H\F$, with target spectrum $X_f$, and with the finite
mapping source equal to the sphere spectrum. The ring has
$\pi_0(H\F)=\Z/2$, and each homotopy group of $X_f$ has
$2$-primary torsion of bounded order. Thus the theorem applies
and gives complete convergence to
$\widehat\pi_*X_f=\varprojlim_r\pi_*X_f/2^r\pi_*X_f$.
In particular, the filtration on this abutment is separated.

The Adams spectral sequence of $Y$ itself has $E_2=T_Y$.
For $Y=H$ this follows from freeness of $\A$, and for
$Y=\HZ$ it follows from Lemma~\ref{lem:integral-resolution}.
The support of these groups excludes every differential:
in the integral case a differential out of $(s,s)$ has stem
$-1$, and there is no source of positive stem that could hit
the diagonal. Hence $E_r(Y)=T_Y$ for every $r\geq2$.

First assume that the family is minimal. Put
\[
 U=\Ext(\Sigma^{-1}K),\qquad E=\Ext(H^*X_f),
 \qquad W=\Ext(C),\qquad a=q^*,\quad b=i^*.
\]
Equation~\eqref{eq:BR} gives $bd_2a=D$.
By Lemma~\ref{lem:D-general}, $D$ is injective.
If $a(u)=0$, then $D(u)=bd_2a(u)=0$, so $u=0$.
Thus $a$ is injective in every bidegree. In the long
exact Ext sequence of the third sequence in
\eqref{eq:three-ses}, every connecting map from
$W^{s,t}$ to $U^{s+1,t}$ has image contained in the
kernel of this injective map $a$. Every such
connecting map is consequently zero, and $b$ is
surjective in every bidegree. This proves
\[
 0\longrightarrow U\xrightarrow{a}E
      \xrightarrow{b}W\longrightarrow0.
\]

Apply Lemma~\ref{lem:acyclic-cancellation} to $d=d_2$.
The quotient $E/S$ has underlying bigraded vector
space
\begin{equation}\label{eq:minimal-Q}
 Q=T_Y\oplus\bigoplus_{d\geq1}\Sigma^{1,d}J_d^\vee.
\end{equation}
If $Y=H$, its support is confined to filtrations zero
and one. The induced differential of bidegree $(2,1)$
on $Q$ is therefore zero.

Suppose $Y=\HZ$. The map $p:X_f\to Y$ induces a map
of Adams spectral sequences. On $E_2$ it is the
composite
\[
 E\xrightarrow{b}W\longrightarrow\Ext(B_Y).
\]
It annihilates $aU$ because $ba=0$, and it
annihilates $d_2(aU)$ because its image under $b$ is
$D(U)$, which maps to zero in $\Ext(B_Y)$ by
\eqref{eq:CI-exact}. Hence it induces a map of
cochain complexes $Q\to T_Y$; on the summand $T_Y$
this map is an isomorphism.

The only possible nonzero differential of bidegree
$(2,1)$ on \eqref{eq:minimal-Q} would send a class
in $(1,2)$ to the class in $(3,3)$ on the tower.
There are no other possibilities: an outgoing
differential from $(s,s)$ has negative stem,
and a target of filtration at least three in $Q$
must be on the tower. The source in $(1,2)$ maps
to zero in $T_Y$, whose degree-$(1,2)$ term is
zero. Naturality forces the image of its
differential in $T_Y^{3,3}$ to be zero. Since
$Q^{3,3}\to T_Y^{3,3}$ is injective, the
differential itself is zero. This proves that
$E_3(X_f)$ is \eqref{eq:minimal-Q}.

For $r\geq3$, a differential out of a
filtration-one class could only land on the
zero-stem tower. The target bidegree is
$(1+r,d+r-1)$, which lies on the tower only if
$d=2$. The same naturality argument excludes
that differential, inductively on $r$, since
the map to $T_Y$ remains an isomorphism on the
tower. An incoming differential to filtration
one would require a negative source filtration.
The tower has no possible outgoing target.
For $Y=H$, the support in filtrations zero and
one excludes every later differential directly.
Thus the minimal-family spectral sequence
collapses at $E_3$.

For an arbitrary family, apply
Theorem~\ref{thm:splitting}. Every factor
$\Sigma^{d-1}H$ in $\Omega G_T$ has free
cohomology $\Sigma^{d-1}\A$ and exactly one
Adams class, in $(0,d-1)$. Its spectral
sequence has no differentials. The product is
equivalent to the locally finite wedge, so
these statements hold degree by degree for
all factors together. The inclusions and projections of
$X_S\times\Omega G_T$ commute with the Adams differentials
and split its spectral sequence in every bidegree. In
particular, there are no differentials between the two
factors. Adding their answers proves
\eqref{eq:general-e3} and the collapse for
the full family.

For $d\geq2$, both $\pi_dY$ and
$\pi_{d-1}Y$ are zero. The connecting map
$\pi_dG\to\pi_{d-1}X_f$ is therefore an
isomorphism. Evaluation on the degree-$d$
cohomology coordinates identifies $\pi_dG$
with $V_d^\vee$. The minimal-fiber summand
has all its positive-degree homotopy in
filtration one, by the computed $E_\infty$,
and the loop-spectrum summand has filtration
zero. Complete convergence gives a separated
filtration. Since no associated-graded term
in positive homotopy lies in filtration at
least two, separatedness implies
$\Fil^2\pi_{d-1}X_f=0$.

To identify $\Fil^1$ in the original
coordinates, reduce \eqref{eq:coefficients}
modulo positive-degree multiples. In degree
$d$, the vectors
\[
 e_r+\sum_{s:\,d_s=d}\epsilon(a_{rs})e_s,
 \qquad r\notin S,\quad d_r=d,
\]
belong to $R_d$. Their distinct redundant
coordinates make them independent, and
their number is $\dim R_d$, so they form
a basis of $R_d$. On homotopy in degree
$d$, the map $L$ is given only by the
coefficients with $d_s=d$: factors in
other suspension degrees have zero
degree-$d$ homotopy. Pulling back the
minimal homotopy summand through the
involution $T$ therefore selects exactly
the functionals that annihilate the
displayed basis of $R_d$. This subspace
is $R_d^\perp$. Dualizing
$0\to R_d\to V_d\to J_d\to0$ identifies
it with the image of $J_d^\vee$, proving
\eqref{eq:canonical-filtration}.
\end{proof}

The preceding theorem specifies the positive-degree filtration
inside the original coordinate homotopy groups. The degree-zero
abelian group can contain a nonsplit extension, which we now
record separately.

\begin{corollary}\label{cor:zero-stem}
Let $f$ satisfy the hypotheses of Theorem~\ref{thm:general-adams}.
If $Y=\HZ$, then the map $X_f\to\HZ$ identifies
\[
 \widehat\pi_0X_f=\Z_2,
 \qquad \Fil^s\widehat\pi_0X_f=2^s\Z_2
 \quad(s\geq0).
\]
If $Y=H$, put $m=\dim_{\F} V_1$ and $j=\dim_{\F} J_1$.
Then $j\in\{0,1\}$, the group $\pi_0X_f$ is already
$2$-complete, and
\[
 \pi_0X_f\cong
 \begin{cases}
  (\Z/2)^{m+1},&j=0,\\
  \Z/4\oplus(\Z/2)^{m-1},&j=1.
 \end{cases}
\]
In the first case $\Fil^1\pi_0X_f=0$. In the second case
$\Fil^1\pi_0X_f=2\pi_0X_f$ and $\Fil^2\pi_0X_f=0$.
\end{corollary}

\begin{proof}
For the integral source, the target $G$ has zero homotopy
in degrees zero and one. The homotopy long exact sequence
therefore makes $\pi_0X_f\to\pi_0\HZ=\Z$ an isomorphism,
and completion gives the displayed identification with
$\Z_2$. The computed $E_\infty$ has exactly one class in
each bidegree $(s,s)$. Consequently every quotient
$\Fil^s/\Fil^{s+1}$ in stem zero has order two. In
particular, $[\Z_2:\Fil^s]=2^s$ for each finite $s$.
These facts determine the subgroups inductively.
Indeed, if $\Fil^s=2^s\Z_2$, the quotient by
$\Fil^{s+1}$ has exponent two, so
$2^{s+1}\Z_2\subseteq\Fil^{s+1}$.
Both subgroups have index two in $2^s\Z_2$, and hence
they are equal. The induction begins with
$\Fil^0=\Z_2$ and proves the integral assertion.

Now assume $Y=H$. The degree-one part of $I$ is a
subspace of $\A^1=\F\{\Sq^1\}$, and
$(\A^+I)^1=0$. Therefore $J_1=I^1$ and $j\in\{0,1\}$.
Choose the minimal subfamily $S$ in
Theorem~\ref{thm:splitting}. Its degree-one target has
$j$ factors, and the redundant target has $m-j$ such
factors. All other factors of $\Omega G_T$ are in
positive suspension degree. Thus the splitting gives
\[
 \pi_0X_f\cong\pi_0X_S\oplus(\Z/2)^{m-j}.
\]
If $j=0$, the target of $f_S$ has zero homotopy in
degrees zero and one, and its homotopy long exact
sequence gives $\pi_0X_S=\Z/2$.

If $j=1$, the degree-one component of $f_S$ is $\Sq^1$.
Projection to that target factor induces a map
\[
 X_S\longrightarrow\fib(H\xrightarrow{\Sq^1}\Sigma H).
\]
The coefficient extension
$0\to\Z/2\to\Z/4\to\Z/2\to0$ gives a fiber sequence
\[
 H(\Z/4)\longrightarrow H\xrightarrow{\beta}\Sigma H.
\]
The map $\beta$ is nonzero: a null map would have fiber
$H\times H$ and hence degree-zero homotopy $(\Z/2)^2$,
whereas this fiber has degree-zero homotopy $\Z/4$.
Since $[H,\Sigma H]=\A^1$ has the single nonzero
element $\Sq^1$, it follows that $\beta=\Sq^1$.
The projection above induces an isomorphism on
$\pi_1$ of the targets and is the identity on
$\pi_0H$. The two degree-zero fiber groups fit into
short exact sequences with those groups as kernel
and quotient. The induced map between the sequences
is therefore an isomorphism on the middle groups,
so $\pi_0X_S=\Z/4$. This proves both formulas for
$\pi_0X_f$, which is finite in either case and hence
already $2$-complete.

In stem zero, the formula for $E_\infty$ has no
positive-filtration term if $j=0$. Separatedness
then gives $\Fil^1=0$. If $j=1$, its only
positive-filtration term is a single copy of $\F$
in filtration one. Thus $\Fil^2=0$ by separatedness,
and $\Fil^1$ has order two. The quotient
$\pi_0X_f/\Fil^1$ is a vector space over $\F$, so
$2\pi_0X_f\subseteq\Fil^1$. The computed abelian
group has $2\pi_0X_f$ of order two. The inclusion
is therefore an equality, as asserted.
\end{proof}

\begin{corollary}\label{cor:subfamilies}
Let $f':Y\to G'$ be a subfamily of a family $f$
as in Theorem~\ref{thm:general-adams}, and let
$X_f\to X_{f'}$ be induced by the target projection.
Then the inclusions $V'\to V$ and $I'\to I$
induce maps $R'\to R$ and $J'\to J$. For every
$d\geq2$, the induced maps on the two possibly nonzero
associated-graded parts of $\pi_{d-1}$ are
\[
 \gr^1\pi_{d-1}X_f\longrightarrow
      \gr^1\pi_{d-1}X_{f'}:
      J_d^\vee\longrightarrow(J'_d)^\vee,
\]
and
\[
 \gr^0\pi_{d-1}X_f\longrightarrow
      \gr^0\pi_{d-1}X_{f'}:
      R_d^\vee\longrightarrow(R'_d)^\vee,
\]
given by the dual maps.
\end{corollary}

\begin{proof}
The inclusion $I'\to I$ sends $\A^+I'$
into $\A^+I$, so it induces $J'\to J$.
An element of $R'$ maps to zero in $J'$
and hence to zero in $J$, so its image
in $V$ lies in $R$. Naturality of the
homotopy connecting isomorphisms
identifies the map on positive
homotopy with $V_d^\vee\to(V'_d)^\vee$.
The subspaces in
\eqref{eq:canonical-filtration} are
the images of the dual maps from $J$.
Restricting to those subspaces gives
the first assertion. Passing to their
quotients and dualizing the inclusions
of $R'$ into $R$ gives the second.
\end{proof}

\section{Families of square operations}\label{sec:examples}

The general theorem expresses the answer through relative minimal
generators. We now determine those generators for the square families
that motivated the calculation. Individual squares give Bruner's first
two examples \cite{Bruner2026}, while the triangular recursion produces a stable
splitting for the infinite even-square family and every finite
truncation. A pair of finite examples shows how the filtration changes
when additional operations make an existing component redundant. All
displayed Adams-page decompositions in this section are isomorphisms
of bigraded vector spaces. The following formulas recover the results of
\cite[Theorems~3.1 and~4.1]{Bruner2026}. 

\begin{corollary}\label{cor:single-square-fibers}
Let
\[
F_n=\fib(H\xrightarrow{\Sq^n}\Sigma^nH)
\quad(n>0),
\qquad
F_{n\Z}=\fib(\HZ\xrightarrow{\Sq^n\rho}\Sigma^nH)
\quad(n\geq2),
\]
where $\rho:\HZ\to H$ is reduction modulo $2$. Then
\[
E_3^{*,*}(F_n)=E_\infty^{*,*}(F_n)
\cong\F\oplus\Sigma^{1,n}\F
\]
and
\[
E_3^{*,*}(F_{n\Z})=E_\infty^{*,*}(F_{n\Z})
\cong\F[h_0]\oplus\Sigma^{1,n}\F.
\]
\end{corollary}

\begin{proof}
For the mod-$2$ source, $\Sq^n$ is a nonzero admissible
basis element of $\A$. For the integral source with $n\geq2$,
its image in $B=\A/\A\Sq^1$ is nonzero: the single admissible
word $\Sq^n$ does not end in $\Sq^1$, whereas
Lemma~\ref{lem:integral-resolution} identifies the ideal
$\A\Sq^1$ with the span of the admissible words that do end
in $\Sq^1$.

In either case the image module $I$ is generated by one
nonzero homogeneous element of degree $n$ and has no
components below degree $n$. Its submodule $\A^+I$ has no
degree-$n$ component. Therefore $J=I/\A^+I$ is
one-dimensional in degree $n$, and the map $V\to J$ is
an isomorphism. In particular, $R=0$. Substitution in
Theorem~\ref{thm:general-adams} gives the formulas and
the collapse assertions. This argument uses minimality
inside the image module and applies equally when $\Sq^n$
is decomposable as an element of $\A$.
\end{proof}

\begin{remark}
For $n=1$ on the integral source, $\Sq^1\rho$ represents
the zero element of $H^1(\HZ;\F)$ and is null-homotopic.
Its fiber is consequently $\HZ\times H$, with Adams
$E_3$ equal to $\F[h_0]\oplus\F$, the additional summand
being in bidegree $(0,0)$. This is the reason for the
integral hypothesis $n\geq2$ in
Corollary~\ref{cor:single-square-fibers}.
\end{remark}

\begin{remark}
For the mod-$2$ source, the case $n=1$ remains in
Corollary~\ref{cor:single-square-fibers}. The identification
$F_1\simeq H(\Z/4)$ and its filtration are established
in the proof of Corollary~\ref{cor:zero-stem}. Its two
Adams $E_\infty$ classes are associated-graded pieces
of $\Z/4$, not a splitting of that abelian group.
\end{remark}

Let
\[
\mathcal N=\{m\geq2:m\text{ is even and is not a power of }2\},
\]
and define
\[
F_{\mathrm{even}}=
\fib\left(\HZ\xrightarrow{(\Sq^{2i}\rho)_{i>0}}
\prod_{i>0}\Sigma^{2i}H\right),
\qquad
F_{\mathrm{pow}}=
\fib\left(\HZ\xrightarrow{(\Sq^{2^j}\rho)_{j>0}}
\prod_{j>0}\Sigma^{2^j}H\right).
\]
We also write $F=F_{\mathrm{even}}$, in Bruner's notation.
For $m\in\mathcal N$, retain the coefficients $a_{m,j}$
constructed by Proposition~\ref{prop:square-recursion}.

\begin{theorem}\label{thm:even-splitting}
Let $F_{\mathrm{even}}$, $F_{\mathrm{pow}}$, and
$\mathcal N$ be as above. Then there exists a stable
equivalence
\begin{equation}\label{eq:even-fiber-splitting}
F_{\mathrm{even}}\simeq
F_{\mathrm{pow}}\times
\prod_{m\in\mathcal N}\Sigma^{m-1}H.
\end{equation}
Let
\[
G_{\mathrm{pow}}=\prod_{j>0}\Sigma^{2^j}H,
\qquad
G_{\mathcal N}=\prod_{m\in\mathcal N}\Sigma^mH.
\]
The equivalence \eqref{eq:even-fiber-splitting} is induced
by the target involution
\[
T:G_{\mathrm{pow}}\times G_{\mathcal N}
\longrightarrow G_{\mathrm{pow}}\times G_{\mathcal N},
\qquad
T(z,w)=(z,w+Lz),
\]
whose $m$th component on the second factor is determined by
\[
L_m=
\sum_{\substack{j>0\\2^j<m}}
\bigl(\Sigma^{2^j}a_{m,j}\bigr)\operatorname{pr}_j.
\]
Moreover,
\begin{equation}\label{eq:power-fiber-e3}
E_3^{*,*}(F_{\mathrm{pow}})
=E_\infty^{*,*}(F_{\mathrm{pow}})
\cong\F[h_0]\oplus\bigoplus_{j>0}\Sigma^{1,2^j}\F.
\end{equation}
\end{theorem}

\begin{proof}
Equation~\eqref{eq:triangular-B} gives, for each
$m\in\mathcal N$,
\[
\overline{\Sq^m}=
\sum_{\substack{j>0\\2^j<m}}
a_{m,j}\overline{\Sq^{2^j}}
\quad\text{in }B.
\]
Since $B^m=[\HZ,\Sigma^mH]$, this equality is precisely
an equality of the represented stable maps
\[
\Sq^m\rho\simeq
\sum_{\substack{j>0\\2^j<m}}
\bigl(\Sigma^{2^j}a_{m,j}\bigr)\Sq^{2^j}\rho.
\]
The term with rightmost factor $\Sq^1$ in the recursion
vanishes here because $\Sq^1\rho$ is null-homotopic.
Each displayed sum is finite. The formulas therefore
define the components of $L$, and hence the map $L$
into the product.

The power-of-two operations generate $B^{>0}$, as
proved after Proposition~\ref{prop:square-recursion}.
Their classes form a basis of
$B^{>0}/\A^+B^{>0}$ by the same argument. Thus they
are a minimal generating family of the image of the
full even-square row. The hypotheses of
Theorem~\ref{thm:splitting} hold with this minimal
subset and with the displayed coefficient expressions.
That theorem gives the target involution and the
stable equivalence.

For the power-of-two family, $J$ is one-dimensional
in each degree $2^j$, $j>0$, and is zero in all other
degrees; also $R=0$. Formula~\eqref{eq:power-fiber-e3}
now follows from Theorem~\ref{thm:general-adams}.
\end{proof}

The following consequence is Bruner's infinite-family
calculation \cite[Theorem~5.5]{Bruner2026}.

\begin{corollary}\label{cor:even-e3}
Let $F_{\mathrm{even}}$ be as above. Then
\[
\begin{aligned}
E_3^{*,*}(F_{\mathrm{even}})
&=E_\infty^{*,*}(F_{\mathrm{even}})\\
&\cong\F[h_0]
\oplus\bigoplus_{j>0}\Sigma^{1,2^j}\F
\oplus\bigoplus_{m\in\mathcal N}\Sigma^{0,m-1}\F.
\end{aligned}
\]
\end{corollary}

\begin{proof}
Apply the Adams spectral sequence to
\eqref{eq:even-fiber-splitting}. Its minimal-fiber
factor contributes \eqref{eq:power-fiber-e3}.
Each factor $\Sigma^{m-1}H$ contributes exactly one
class in $(0,m-1)$, since its cohomology is the free
module $\Sigma^{m-1}\A$. There are finitely many
contributing factors in each bidegree. The stable
splitting therefore gives the displayed direct sum
degreewise, and each factor has zero differentials
from $E_3$ onward.
\end{proof}

\begin{corollary}\label{cor:finite-even-truncations}
Let $N\geq2$ be an integer, and define
\[
F_{\mathrm{even},\leq N}=
\fib\left(\HZ\xrightarrow{(\Sq^m\rho)_{2\leq m\leq N,\ m\text{ even}}}
\prod_{\substack{2\leq m\leq N\\m\text{ even}}}\Sigma^mH\right),
\]
\[
F_{\mathrm{pow},\leq N}=
\fib\left(\HZ\xrightarrow{(\Sq^{2^j}\rho)_{j>0,\ 2^j\leq N}}
\prod_{\substack{j>0\\2^j\leq N}}\Sigma^{2^j}H\right).
\]
Then
\[
F_{\mathrm{even},\leq N}\simeq
F_{\mathrm{pow},\leq N}\times
\prod_{\substack{m\in\mathcal N\\m\leq N}}\Sigma^{m-1}H,
\]
and
\[
\begin{aligned}
E_3^{*,*}(F_{\mathrm{even},\leq N})
&=E_\infty^{*,*}(F_{\mathrm{even},\leq N})\\
&\cong\F[h_0]
\oplus\bigoplus_{\substack{j>0\\2^j\leq N}}\Sigma^{1,2^j}\F
\oplus\bigoplus_{\substack{m\in\mathcal N\\m\leq N}}\Sigma^{0,m-1}\F.
\end{aligned}
\]
\end{corollary}

\begin{proof}
The expression for a redundant component of degree
$m$ uses only powers $2^j<m$. If $m\leq N$, every
one of these coordinates is retained in the finite
power-of-two target. The involution in
Theorem~\ref{thm:even-splitting} consequently restricts
to the displayed finite products and produces the
stated splitting.

The retained power-of-two operations generate the
image of the finite even row by the same triangular
expressions. They remain minimal: a relation
expressing one of them by positive-degree multiples
of the others would, after lifting from $B$ to
$\A$, contradict its indecomposability in
$\A^+/(\A^+)^2$. The finite minimal family therefore
has one generator of $J$ in each retained power-of-two
degree and has $R=0$. Applying
Theorem~\ref{thm:general-adams} to this family and
adding the suspended $H$-factors proves the formula.
The image of a finite row can be a proper submodule
of $B^{>0}$; the proof uses its own quotient by
positive-degree multiples.
\end{proof}

\begin{proposition}\label{prop:relative-square-pairs}
Let
\[
X_{4,6}=\fib\bigl(\HZ\xrightarrow{(\Sq^4\rho,\Sq^6\rho)}
\Sigma^4H\times\Sigma^6H\bigr).
\]
Then
\[
X_{4,6}\simeq F_{4\Z}\times\Sigma^5H,
\qquad
E_3^{*,*}(X_{4,6})=E_\infty^{*,*}(X_{4,6})
\cong\F[h_0]\oplus\Sigma^{1,4}\F\oplus\Sigma^{0,5}\F.
\]
The projection $X_{4,6}\to F_{6\Z}$ is an isomorphism
on $\pi_5$ and takes its nonzero filtration-zero
class to a filtration-one class.
\end{proposition}

\begin{proof}
The Adem relation with $(a,b)=(2,4)$ is
\[
\Sq^2\Sq^4=\Sq^6+\Sq^5\Sq^1.
\]
After composition with $\rho$, its last term is zero.
Thus the degree-six operation is generated by the
degree-four operation, and its coefficient is the
positive-degree operation $\Sq^2$. The degree-four
operation is nonzero and is a minimal generator
because the image has no components below degree
four. Consequently $J$ is one-dimensional in degree
four, and $R$ is one-dimensional in degree six.
The target involution is
\[
(z_4,z_6)\longmapsto
(z_4,z_6+(\Sigma^4\Sq^2)z_4).
\]
Theorems~\ref{thm:splitting} and
\ref{thm:general-adams} now give the stable splitting
and the Adams-page formula.

Projection onto the degree-six target factor is an
isomorphism on its degree-six homotopy group. The
homotopy connecting maps therefore identify the
induced map $\pi_5X_{4,6}\to\pi_5F_{6\Z}$ with the
identity of $\F$. The source class is in filtration
zero by the computed formula, whereas the nonzero
class of $F_{6\Z}$ is in filtration one by
Corollary~\ref{cor:single-square-fibers}. This proves
the final assertion.
\end{proof}

\begin{remark}
The lower-degree operation available in a family is
essential to this calculation. For the pair
$(\Sq^2\rho,\Sq^6\rho)$, both components are minimal.
Indeed, $\A^4$ has admissible basis
$\Sq^4,\Sq^3\Sq^1$. Their products with $\Sq^2$,
reduced modulo $\A\Sq^1$, are respectively
$\overline{\Sq^4\Sq^2}$ and zero. For the second
product, the relations
$\Sq^1\Sq^2=\Sq^3$ and
$\Sq^3\Sq^3=\Sq^5\Sq^1$ give the claimed vanishing.
The admissible words $\Sq^6$ and $\Sq^4\Sq^2$
represent distinct basis elements of $B^6$.
Therefore $\overline{\Sq^6}$ is not a left multiple
of $\overline{\Sq^2}$ in degree six. Theorem
\ref{thm:general-adams} gives, for this pair,
\[
E_3=E_\infty\cong
\F[h_0]\oplus\Sigma^{1,2}\F\oplus\Sigma^{1,6}\F.
\]
\end{remark}

\subsection{Secondary Ext by comparison}

Let $\B$ be the secondary Steenrod algebra, let
$\HB X$ denote the secondary cohomology of a spectrum,
and put
\[
 \mathbb S_{\B}=\HB\mathbb S,
\]
where $\mathbb S$ is the sphere spectrum. The notation
$\sExt$ denotes the secondary derived functor of
Baues--Jibladze, with the bigrading inherited from
ordinary Ext and its secondary differential.
The coefficient object $\mathbb S_{\B}$ is weakly
equivalent to the standard augmentation pair module
\[
 \mathbb G^\Sigma=
 \left(\F\oplus\Sigma\F\xrightarrow{\partial}\Z/4\right),
 \qquad\partial(a,b)=2a,
\]
with its specified $\B$-module action
\cite[Example~2.2.2]{BauesJibladze2011}.
In particular, its invariants are
$\pi_0\mathbb S_{\B}=\F$ and
$\pi_1\mathbb S_{\B}=\Sigma\F$. Specifying the
ordinary module $\F$ alone does not specify this
secondary coefficient object.

\begin{corollary}\label{cor:secondary-comparison}
Let $f$ and $X_f$ satisfy the hypotheses of
Theorem~\ref{thm:general-adams}. Then there is an
isomorphism of bigraded vector spaces
\[
 \sExt^{s,t}(\HB X_f,\mathbb S_{\B})
       \cong E_3^{s,t}(X_f),
\]
where the right side is given by
\eqref{eq:general-e3}. In particular, the secondary
Ext groups of $F_n$, $F_{n\Z}$, and $F$
are the bigraded vector spaces in
Corollaries~\ref{cor:single-square-fibers} and
\ref{cor:even-e3}, respectively.
\end{corollary}

\begin{proof}
The cohomology of $X_f$ is bounded below and
finite dimensional in every degree: these
properties hold for $P$ and $B_Y$ and hence for
the kernels, cokernels, and extension in
\eqref{eq:three-ses}. Thus $X_f$ is of finite
type for the comparison theorem, and the
mapping source $\mathbb S$ is finite
dimensional. Apply
\cite[Theorem~3.1.1]{BauesJibladze2011}, with
mapping target $X_f$ and mapping source
$\mathbb S$. This gives the displayed
isomorphism. The value of its right side is
Theorem~\ref{thm:general-adams}; the specified
special cases are the cited corollaries.
\end{proof}

\textit{Remark.} The computation successfully determines the secondary Ext groups by applying the Baues--Jibladze comparison theorem \cite{BauesJibladze2011}, extending Bruner's observation \cite[Section~6]{Bruner2026} to the general operation families of Theorem~\ref{thm:general-adams}. However, this approach derives the secondary Ext groups from the Adams spectral sequence, rather than constructing the secondary cohomology object $\HB X_f$ directly from a lifted presentation of the primary operation row. A natural question is whether one can explicitly build this secondary module and its resolution to provide a direct proof (cf.\ \cite[Questions~6.1 and~6.2]{Bruner2026}). Sections \ref{sec:tracks} and \ref{sec:augmentation} below address this by demonstrating why such a direct lifting strategy is not merely a matter of choosing null homotopies, but rather encounters fundamental structural obstructions.

\section{Exact secondary lifts and their coherence obstruction}\label{sec:tracks}

The preceding calculations determine the Adams filtration without
constructing the secondary cohomology object of the fiber. We now
specify the algebraic data that such a construction would require.
The first step is an exact boundary equation in a crossed bimodule.
The next step is a coherence equation, whose obstruction we describe
in a Hom complex. Keeping these steps separate prevents a primary
identity from being used as an unproved secondary identity.

\subsection{Crossed bimodules and sections}

Let $B_0$ be an associative unital ring, let $B_1$ be a
$B_0$-bimodule, and let $\partial:B_1\to B_0$ be a
bimodule homomorphism satisfying
\begin{equation}\label{eq:Peiffer}
 (\partial x)y=x(\partial y)\qquad(x,y\in B_1).
\end{equation}
These data constitute a crossed bimodule. Put
$A=\coker\partial$, $M=\ker\partial$, and let
$p:B_0\to A$ be the quotient map. The ideal
$\im\partial$ annihilates $M$ on both sides, since
\[
 (\partial x)m=x(\partial m)=0,
 \qquad m(\partial x)=(\partial m)x=0
 \quad(x\in B_1,\ m\in M).
\]
Consequently $M$ is an $A$-bimodule. All the constructions below
also apply to graded rings and bimodules with degree-preserving
structure maps; in that case all sections are homogeneous and all
Hom groups consist of maps of internal degree zero.

The secondary Steenrod algebra has
$A=\A$ and $M\cong\Sigma\A$. Its structure as a pair
algebra and the associated track category are treated in
\cite{Baues2006,BauesJibladze2011}. A grading conversion is needed
when using Nassau's notation. His differential
$D_1\to D_0$ raises internal degree by one and has kernel
$\A$; replacing $D_1$ by $\Sigma D_1$ makes the
differential degree preserving and its kernel $\Sigma\A$
\cite[Section~1]{Nassau2012}. Formulas from the two conventions
must be translated before they are combined.

Choose set-theoretic sections
\[
 \sigma:A\longrightarrow B_0,
 \qquad u:\im\partial\longrightarrow B_1,
 \qquad p\sigma=\id,\quad \partial u=\id,
 \quad\sigma(0)=u(0)=0.
\]
No additivity or multiplicativity is imposed on these sections.

\begin{lemma}\label{lem:exact-lift}
Let $\partial:B_1\to B_0$, $p$, $\sigma$, and $u$
be as above. If $x\in B_0$, then
\begin{equation}\label{eq:exact-track}
 T(x)=u\bigl(x-\sigma(p(x))\bigr)
 \quad\text{satisfies}\quad
 \partial T(x)=x-\sigma(p(x)).
\end{equation}
If $p(x)=0$, then $\partial T(x)=x$. If
$x=\sigma(p(x))$, then $T(x)=0$. If, in addition,
the crossed bimodule is graded, $B_0^0=\Z/4$, $A^0=\F$,
and $p$ in degree zero is reduction modulo $2$, then no
additive section of $p$ exists.
\end{lemma}

\begin{proof}
Applying $p$ to $x-\sigma(p(x))$ gives zero, so this
difference belongs to $\ker p=\im\partial$ and $u$
is defined on it. The equality $\partial u=\id$ gives
\eqref{eq:exact-track}. If $p(x)=0$, the normalization
$\sigma(0)=0$ gives the second assertion. If
$x=\sigma(p(x))$, the argument of $u$ is zero, so
$u(0)=0$ gives the third assertion.

For the final assertion, suppose an additive section exists.
Restrict it to $A^0=\F$ and compose with the projection
onto $B_0^0=\Z/4$. Since $p$ preserves degree, this gives
an additive section $s_0:\F\to\Z/4$ of reduction modulo
two, even if the proposed original section is not homogeneous.
The value $v=s_0(1)$ is an odd residue class. Additivity gives
$2v=s_0(1+1)=s_0(0)=0$, whereas every odd residue class in
$\Z/4$ satisfies $2v=2$. This contradiction proves
nonexistence.
\end{proof}

The sign in \eqref{eq:exact-track} is part of the boundary
identity. For example, in degree zero take
$x=\sigma(1)\in\Z/4$. Then $T(x)=0$, whereas
$x+\sigma(p(x))=2x=2$. Thus replacing the difference
by a sum does not give a valid identity in characteristic four.
Reducing the proposed identity modulo two would erase this
failure.

For a primary relation between columns, the input to
\eqref{eq:exact-track} is the actual composite of the chosen
column lifts. More precisely, let
$D_1:\widetilde P_1\to\widetilde P_0$ lift a primary
row $d_1:P_1\to P_0$, and let $\rho\in\ker d_1$.
Choose a lift $\widetilde\rho\in\widetilde P_1$.
Each coefficient of $D_1\widetilde\rho$ lies in
$\im\partial$; applying $u$ to these coefficients
gives a preimage of that composite. In an integral
presentation, the column lifting $\Sq^1$ remains in
this calculation. Replacing a lifted product by a chosen
lift of its primary value changes the input and requires
the correction in \eqref{eq:exact-track}.

\subsection{Lifts of a primary complex}

Let $(P_\bullet,d)$ be a chain complex of free left
$A$-modules in nonnegative chain degrees. Choose a basis
in every degree and let $\widetilde P_n$ be the free
left $B_0$-module on the same basis. Choose $B_0$-linear
maps
\[
 D_n:\widetilde P_n\longrightarrow\widetilde P_{n-1}
 \qquad(n\geq1)
\]
lifting $d_n$. Define
\[
 \partial_n:B_1\otimes_{B_0}\widetilde P_n
       \longrightarrow\widetilde P_n,
 \qquad x\otimes e\longmapsto(\partial x)e.
\]
Using the chosen basis identifies $\partial_n$ with a
direct sum of copies of $\partial$. Therefore
\begin{equation}\label{eq:free-kernel}
 \ker\partial_n=M\otimes_A P_n.
\end{equation}
Every coefficient of $D_{n-1}D_n$ maps to zero in $A$.
Choose a preimage under $\partial$ for each such
coefficient. The values on the basis extend uniquely to
a $B_0$-linear map
\begin{equation}\label{eq:null-composite}
 h_n:\widetilde P_n\longrightarrow
       B_1\otimes_{B_0}\widetilde P_{n-2},
 \qquad
 \partial_{n-2}h_n=D_{n-1}D_n
 \quad(n\geq2).
\end{equation}
The extension is linear because the domain is free; the
section used to choose the basis values need not be linear.

These maps are tracks between the corresponding free pair
module maps. To check the other track equation, let $e$ be
a basis vector and write $h_n(e)=\sum_i y_i\otimes e_i$.
For $x\in B_1$, the two values that must agree are
\[
 h_n\partial_n(x\otimes e)
       =\sum_i(\partial x)y_i\otimes e_i,
 \qquad
 (1\otimes D_{n-1}D_n)(x\otimes e)
       =\sum_i x(\partial y_i)\otimes e_i.
\]
They agree term by term by \eqref{eq:Peiffer}. The
remaining secondary-complex condition is
\begin{equation}\label{eq:coherence}
 (1\otimes D_{n-2})h_n=h_{n-1}D_n
 \qquad(n\geq3).
\end{equation}
Equation~\eqref{eq:null-composite} does not assert
\eqref{eq:coherence}.

\subsection{The obstruction class}

For $r\geq0$, define
\[
 \mathcal C^r=\prod_{n\geq r}
       \Hom_A(P_n,M\otimes_A P_{n-r}).
\]
Components with negative indices are zero. For
$\phi\in\mathcal C^r$, set
\begin{equation}\label{eq:Hom-differential}
 (\delta\phi)_n=(1_M\otimes d_{n-r})\phi_n
       -(-1)^r\phi_{n-1}d_n.
\end{equation}
Writing $\bar d=1_M\otimes d$, substitution twice gives
\[
 \begin{aligned}
 (\delta^2\phi)_n={}&
 \bar d_{n-r-1}\bar d_{n-r}\phi_n
 -(-1)^r\bar d_{n-r-1}\phi_{n-1}d_n\\
 &-(-1)^{r+1}\bar d_{n-r-1}\phi_{n-1}d_n
 +(-1)^{2r+1}\phi_{n-2}d_{n-1}d_n=0.
 \end{aligned}
\]
The first and last terms vanish because $d^2=0$;
the two middle terms cancel. Thus $\mathcal C^\bullet$
is a cochain complex.

\begin{theorem}\label{thm:coherence-obstruction}
Let $\partial:B_1\to B_0$, $(P_\bullet,d)$, $D_n$,
and $h_n$ satisfy \eqref{eq:Peiffer} and
\eqref{eq:null-composite}. For $n\geq3$, define
\begin{equation}\label{eq:obstruction}
 o_n=(1_{B_1}\otimes D_{n-2})h_n-h_{n-1}D_n.
\end{equation}
Then $o_n$ factors uniquely through an $A$-linear map
$P_n\to M\otimes_A P_{n-3}$. The resulting
$o\in\mathcal C^3$ is a cocycle, and its class
\[
 [o]\in H^3(\mathcal C^\bullet)
\]
is independent of the choices of $h_n$ for the fixed
maps $D_n$. This class is zero if and only if there
are maps $h'_n$ satisfying
\[
 \partial_{n-2}h'_n=D_{n-1}D_n\quad(n\geq2),
 \qquad
 (1_{B_1}\otimes D_{n-2})h'_n=h'_{n-1}D_n
       \quad(n\geq3).
\]
\end{theorem}

\begin{proof}
Applying $\partial_{n-3}$ to \eqref{eq:obstruction}
gives
\[
 D_{n-2}D_{n-1}D_n-D_{n-2}D_{n-1}D_n=0.
\]
Consequently $o_n$ takes values in
$\ker\partial_{n-3}=M\otimes_A P_{n-3}$.
This target is annihilated by $\im\partial$ on
the left. Since $o_n$ is $B_0$-linear, it vanishes
on $(\im\partial)\widetilde P_n$ and therefore
factors through $P_n$. The quotient map is
surjective, so the factorization is unique, and
its linearity is over $A$.

For $n\geq4$, compute the lift of
$(\delta o)_n$ by substituting \eqref{eq:obstruction}:
\[
 \begin{aligned}
 (\delta o)_n={}&
 (1\otimes D_{n-3})(1\otimes D_{n-2})h_n
 -(1\otimes D_{n-3})h_{n-1}D_n\\
 &+(1\otimes D_{n-3})h_{n-1}D_n
 -h_{n-2}D_{n-1}D_n\\
 ={}&(1\otimes D_{n-3}D_{n-2})h_n
 -h_{n-2}D_{n-1}D_n.
 \end{aligned}
\]
Let $e$ be a basis vector of $\widetilde P_n$ and
write
\[
 h_n(e)=\sum_i x_i\otimes e_i,
 \qquad
 h_{n-2}(e_i)=\sum_j y_{ij}\otimes e'_j.
\]
Every sum is finite because the modules are free
direct sums. The boundary equations in
\eqref{eq:null-composite} identify the value of
the first term on $e$ with
\[
 \sum_{i,j}x_i(\partial y_{ij})\otimes e'_j,
\]
and the value of the second term with
\[
 \sum_{i,j}(\partial x_i)y_{ij}\otimes e'_j.
\]
Equation~\eqref{eq:Peiffer} makes them equal
term by term. Hence $(\delta o)_n=0$ for every
$n\geq4$, which proves that $o$ is a cocycle.

Let $h'_n$ be another choice satisfying the same
boundary equation for every $n\geq2$. The
difference $h'_n-h_n$ takes values in
$\ker\partial_{n-2}$ and, by the same
factorization argument, defines
$\eta_n:P_n\to M\otimes_A P_{n-2}$.
Together these maps form $\eta\in\mathcal C^2$.
Substitution gives
\[
 o'_n-o_n=(1_M\otimes d_{n-2})\eta_n
                  -\eta_{n-1}d_n
             =(\delta\eta)_n.
\]
It follows that $[o']=[o]$.

If the maps $h'_n$ are coherent, then $o'=0$,
so $[o]=0$. Conversely, if $[o]=0$, choose
$\eta\in\mathcal C^2$ with $\delta\eta=o$.
Compose $\eta_n$ with the quotient
$\widetilde P_n\to P_n$ and the inclusion
$M\otimes_A P_{n-2}\hookrightarrow
B_1\otimes_{B_0}\widetilde P_{n-2}$.
This gives a $B_0$-linear map with zero
boundary. Set $h'_n=h_n-\eta_n$, with this
composite understood. The boundaries remain
$D_{n-1}D_n$, and the preceding identity
gives $o'=o-\delta\eta=0$. Thus the new
maps satisfy every required coherence equation.
\end{proof}

The theorem concerns a fixed unaugmented primary complex and
fixed differential lifts. Its sufficiency statement supplies
coherent null-composite tracks for those data; it does not
assert exactness of the primary complex or identify its
augmentation. No vanishing or nonvanishing of the class is
asserted for an unspecified secondary model of a fiber.
In particular, the criterion is not a nonexistence theorem
for all secondary realizations of the same primary module.
Secondary resolutions require the object and augmented
exactness in addition to coherence, as in
\cite[Section~2]{BauesJibladze2006}.

\subsection{The repeated \texorpdfstring{$\Sq^1$}{Sq1} composite}

One of the null tracks relevant to integral
cohomology can be checked explicitly. In this
subsection we use Nassau's unshifted notation
$D_1\to D_0$ and its multiplication rules
\cite[pp.~680--681]{Nassau2012}. Write
$q=\Sq^1$ and
$Q_1=\Sq(\Delta_2)=\Sq^2\Sq^1+\Sq^3\in\A$.

\begin{proposition}\label{prop:Sq1-track}
Let $h=\mu_0\Sq^2+U_{-1,0}\in D_1$.
Then
\[
 \partial h=q^2,
 \qquad qh=hq.
\]
\end{proposition}

\begin{proof}
The product in \cite[Fig.~1, p.~689]{Nassau2012} and the
boundary formulas in \cite[p.~681]{Nassau2012} give
\[
 q^2=2\Sq^2+Y_{-1,0},
 \qquad
 \partial(\mu_0\Sq^2)=2\Sq^2,
 \qquad\partial U_{-1,0}=Y_{-1,0}.
\]
Adding the two boundary values proves
$\partial h=q^2$. The left action formulas give
\[
 q\mu_0=\mu_0q+1,
 \qquad qU_{-1,0}=U_{-1,0}q+U_{0,0},
 \qquad U_{0,0}=\mu_0Q_1+\Sq^2.
\]
The module $D_1$ has characteristic two. Using
these formulas, and taking the coefficients of
$\mu_0$ in $\A$, yields
\[
 \begin{aligned}
 qh-hq
 &=\mu_0(q\Sq^2+\Sq^2q)+\Sq^2+U_{0,0}\\
 &=\mu_0Q_1+\Sq^2+\mu_0Q_1+\Sq^2=0.
 \end{aligned}
\]
The middle equality uses the primary Adem identity
$q\Sq^2=\Sq^3$ and the definition of $Q_1$.
This proves the overlap equation.
\end{proof}

After regrading, this track supplies coherent
null composites for the repeated $\Sq^1$
differential of Lemma~\ref{lem:integral-resolution}.
Its augmentation to a specified secondary module
is still separate data.

\section{The primary augmentation and the remaining secondary problem}\label{sec:augmentation}

The coherence criterion established in Section~\ref{sec:tracks} applies after a primary complex and its lifts have been chosen. To use such a complex for the secondary cohomology of a fiber, its augmentation must have the correct primary cokernel and must satisfy its own track equations. In \cite[Questions~6.1 and~6.2]{Bruner2026}, Bruner asks for a direct computation of the secondary cohomology of these fibers and a subsequent proof of the two-extension theorem via the secondary Adams spectral sequence. Theorem~\ref{thm:coherence-obstruction} in Section \ref{sec:tracks} and the forthcoming Proposition~\ref{prop:augmentation-obstruction} below serve as a direct answer to the limitations of this proposed strategy. We prove that the operation row alone fails the first augmentation requirement whenever its kernel is nonzero, and that changing only higher chain degrees cannot correct this discrepancy. We then state the missing augmentation equation and delimit the conclusion provided by the Baues--Jibladze comparison theorem \cite{BauesJibladze2006}.

\subsection{The cokernel of the operation row}

\begin{proposition}\label{prop:augmentation-obstruction}
Let $f$, $K$, and $C$ be as in
Section~\ref{sec:rows}. If $K\ne0$, then
$C$ and $H^*X_f$ are not isomorphic as graded
vector spaces. If a nonnegative chain complex
has degree-zero homology $C$, then adjoining
generators and differentials only in chain
degrees at least two cannot make it a
resolution of $H^*X_f$.
\end{proposition}

\begin{proof}
Every module in the third exact sequence of
\eqref{eq:three-ses} is finite dimensional in
each internal degree. Exactness therefore gives
\begin{equation}\label{eq:dimension-obstruction}
 \dim_{\F} H^kX_f
       =\dim_{\F} C^k+\dim_{\F} K^{k+1}
       \qquad(k\in\Z).
\end{equation}
Choose $m$ with $K^m\ne0$. Taking $k=m-1$
in \eqref{eq:dimension-obstruction} gives
$\dim H^{m-1}X_f>\dim C^{m-1}$, so no
graded vector-space isomorphism exists.

For a nonnegative complex $Q_\bullet$, its
degree-zero homology is
$Q_0/\im(d_1:Q_1\to Q_0)$. Changing only
$Q_n$ and $d_n$ for $n\geq2$ leaves this
quotient unchanged. A resolution of $H^*X_f$
must have degree-zero homology $H^*X_f$.
The first assertion therefore rules out the
proposed change.
\end{proof}

For $Y=H$, the row $P\to\A$ has cokernel
$C$. For $Y=\HZ$, replacing the quotient
$B=\A/\A\Sq^1$ by a free presentation
gives the total row
\[
 \A e_1\oplus P\longrightarrow\A,
 \qquad e_1\longmapsto\Sq^1,\quad
 e_\lambda\longmapsto\widetilde\theta_\lambda,
\]
where $|e_1|=1$ and $\widetilde\theta_\lambda\in\A^{d_\lambda}$
is a homogeneous representative of $\theta_\lambda\in B$.
Its cokernel is again $C$. Indeed, its image
is exactly the inverse image of $I$ under
$\A\to B$: a representative of an element
of $I$ differs from a sum of the chosen
column multiples by an element of
$\A\Sq^1$. Thus adding the integral
column does not change the obstruction in
Proposition~\ref{prop:augmentation-obstruction}.

\begin{proposition}\label{lem:uniform-row-kernel}
Let $n\geq1$. If
$\nu_n:\Sigma^n\A\to\A$ is given by
$\nu_n(e_n)=\Sq^n$, then
\[
 0\ne\Sq^{2n-1}e_n\in\ker\nu_n.
\]
If $n\geq2$ and the target is replaced
by $B=\A/\A\Sq^1$, then the same
element belongs to the kernel.
\end{proposition}

\begin{proof}
The one-term admissible monomial
$\Sq^{2n-1}$ is nonzero, and
$\Sigma^n\A$ is free on $e_n$.
Thus $\Sq^{2n-1}e_n\ne0$. The Adem
relation gives
\[
 \Sq^{2n-1}\Sq^n
 =\sum_{k=0}^{n-1}
   \binom{n-k-1}{2n-1-2k}\Sq^{3n-1-k}\Sq^k.
\]
For $0\leq k\leq n-1$, the upper
entry $n-k-1$ is nonnegative, while
\[
 (2n-1-2k)-(n-k-1)=n-k>0.
\]
Every binomial coefficient is therefore
zero as an integer. The product vanishes
in $\A$, proving the first kernel
assertion, and its image in $B$ also
vanishes, proving the second.
\end{proof}

For the full even-square family, the
obstruction already occurs in internal
degree four. Let $e_2$ be its degree-two
column. The element $\Sq^2e_2$ is
nonzero in the free source and maps to
\[
 \overline{\Sq^2\Sq^2}
       =\overline{\Sq^3\Sq^1}=0
       \quad\text{in }B.
\]
Its degree-four kernel is exactly the
line spanned by this element. In fact,
the degree-four part of the free source
has basis $\Sq^2e_2,e_4$; the first
maps to zero and the second maps to the
nonzero class $\overline{\Sq^4}$.
For the full family, $C=\F$, since its
image is $B^{>0}$. Equation~\eqref{eq:dimension-obstruction}
then gives
\begin{equation}\label{eq:first-missing-class}
 C^3=0,\qquad H^3F=\F.
\end{equation}
The same equality holds with $F$ replaced
by $F_{2\Z}$ and $C$ by
$B/\A\overline{\Sq^2}$: its degree-three
quotient is zero because
$\Sq^1\Sq^2=\Sq^3$, and its degree-four
kernel is the line spanned by $\Sq^2e_2$.
Thus the cokernel discrepancy affects a
specific cohomology group, even in the
first nontrivial integral example.

\subsection{The augmentation track equation}

Write $\mathscr P_n$ for free pair modules
over a crossed algebra and $N$ for the
pair module to be resolved. An augmented
partial secondary complex begins with
maps
\[
 \mathscr P_1\xrightarrow{D_1}
 \mathscr P_0\xrightarrow{\varepsilon}N
\]
and a track
$\Gamma_0:\varepsilon D_1\Rightarrow0$.
For a free pair module $E$, a degree-one
secondary cycle consists of a map
$\rho:E\to\mathscr P_1$ and a track
$H:D_1\rho\Rightarrow0$ satisfying
\begin{equation}\label{eq:augmentation-cycle}
 \varepsilon H=\Gamma_0\rho.
\end{equation}
Both sides are tracks from
$\varepsilon D_1\rho$ to zero. Equality
of their boundaries does not imply
equality of the tracks: their difference
is an automorphism track, represented
by a map into $\pi_1N$. This is the
lowest coherence condition for the
augmented complex; see
\cite[Definitions~2.10 and~2.12]{BauesJibladze2006}.

Adjoining a generator with differential
$\rho$ and null-composite track $H$
preserves augmented coherence only when
\eqref{eq:augmentation-cycle} holds.
An all-cycle construction must therefore
range over cycles of the specified
augmented secondary complex. The choice of
an unaugmented null track alone does not
verify this requirement. The
existence construction in
\cite[Lemma~2.14]{BauesJibladze2006}
begins with the object to be resolved
and its augmentation data. It cannot
identify that object from the operation
row or change the cokernel computed
in Proposition~\ref{prop:augmentation-obstruction}.

For a fiber $X_f$, a direct secondary
calculation must provide a pair module
weakly equivalent to $\HB X_f$, with
\[
 \pi_0(\HB X_f)=H^*X_f,
 \qquad \pi_1(\HB X_f)\cong\Sigma H^*X_f,
\]
and a secondary free resolution with
its specified augmentation. Its
degree-zero free generators must
surject onto $H^*X_f$, including the
contribution of $\Sigma^{-1}K$ in
\eqref{eq:three-ses}. The null-composite
tracks must satisfy both the coherence
equations of Section~\ref{sec:tracks}
and the augmentation equations.

\subsection{The scope of the comparison calculation}

A secondary free resolution of the actual secondary cohomology object $\HB X_f$ 
defines a secondary algebraic differential $d_{(2)}$ on the primary Ext group $\Ext_{\A}(H^*X_f,\F)$. 
The Baues--Jibladze comparison theorem \cite[Theorem~7.3]{BauesJibladze2006} identifies this 
algebraic differential $d_{(2)}$ with the topological Adams differential $d_2$. 
Combining this identification with the Bruner--Rognes two-extension theorem 
\cite[Theorem~1.1]{BrunerRognes2022}, which evaluates the topological $d_2$, yields the relation
\begin{equation}\label{eq:secondary-BR-component}
 i^*d_{(2)}q^*=D.
\end{equation}
Here, $i^*$ and $q^*$ are the ordinary primary Ext maps induced by the sequences 
in \eqref{eq:three-ses}; crucially, this relation does not require or assume any 
secondary realization of the modules $C$ or $\Sigma^{-1}K$.

Because Equation~\eqref{eq:secondary-BR-component} is derived by substituting the 
Bruner--Rognes topological formula into the Baues--Jibladze comparison isomorphism, 
it merely reflects the known topological reality rather than providing an independent 
algebraic proof of it. As requested in \cite[Question~6.2]{Bruner2026}, providing a purely 
independent secondary proof of the two-extension theorem would require building the 
secondary complex directly. 

\medskip

\textit{Final remark.} The results proved in this paper determine the operation-fiber Adams 
filtrations and compute their secondary Ext groups via stable splittings and comparison 
theorems. Simultaneously, by establishing exact boundary conventions, a coherence 
criterion for fixed differential lifts, and an obstruction to resolving the fiber 
cohomology with the operation-row cokernel, we rigorously delineate the fundamental 
algebraic hurdles that must be overcome to model these fibers directly in the 
secondary Steenrod algebra.

\medskip

\subsection*{Funding}
None.

\subsection*{Declaration of Competing Interest}
The author declares that he has no known competing financial interests or personal relationships that could have appeared to influence the work reported in this paper.

\subsection*{Data Availability}
Data sharing is not applicable because no datasets were generated or analyzed.

\subsection*{Declaration on AI usage}
AI-assisted tools were used for literature retrieval, and assisting with the verification of mathematical arguments. However, the human author retains sole ownership of all core ideas and takes full responsibility for the final content and its accuracy.

\end{document}